\documentclass[12pt]{amsart}
\usepackage{amsthm,vmargin}
\usepackage{amsmath,xypic}
\usepackage{amssymb}
\usepackage{hyperref}
\usepackage{epigraph}
\usepackage{graphicx}

\newtheorem{theorem}[equation]{Theorem}      
\newtheorem{lemma}[equation]{Lemma}          %
\newtheorem{corollary}[equation]{Corollary}  
\newtheorem{proposition}[equation]{Proposition}

\theoremstyle{definition}
\newtheorem{definition}[equation]{Definition}

\theoremstyle{definition}

\theoremstyle{remark}

\theoremstyle{definition}
\newtheorem{remark}[equation]{Remark}

\numberwithin{equation}{section} 

\DeclareMathOperator{\Sym}{Sym}

\newcommand{\GG}{\mathcal{G}}

\newcommand{\NN}{\mathbb{N}}
\newcommand{\ZZ}{\mathbb{Z}}
\newcommand{\Evec}{\mathcal{E}}
\renewcommand{\ss}[1]{s_{#1}}
\newcommand{\scar}{s_{Car}}
\newcommand{\carperiod}{\tilde{\pi}}
\newcommand{\CCc}{\mathbb{C}_\infty}
\renewcommand{\AA}{A}
\renewcommand{\over}[1]{\chi_t(#1)}
\newcommand{\Tate}{\mathbb{T}}
\newcommand{\del}{\triangle}
\newcommand{\EE}{\mathbb{E}}
\newcommand{\Fq}{\mathbb{F}_q}
\newcommand{\frob}{\mathtt{Frob}}
\DeclareMathOperator{\GL}{\mathbf{GL}}
\newcommand{\FF}{\mathcal{F}}
\newcommand{\RadR}{\mathcal{R}}
\newcommand{\DVMF}[5]{\mathcal{M}_{#1, #2}^{#3} (#4, #5)}
\newcommand{\DMF}[2]{M_{#1, #2}}

\newcommand{\tauField}{\mathcal{K}}

\newcommand{\calG}{\mathcal{G}}
\newcommand{\calH}{\mathcal{H}}
\newcommand{\Car}{\text{Car}}

\hypersetup{
    colorlinks,%
    citecolor=black,%
    filecolor=black,%
    linkcolor=blue,%
    urlcolor=black
}

\begin{document}

\title[$\tau$-recurrent sequences and Eisenstein series]{On symmetric powers of $\tau$-recurrent sequences and deformations of Eisenstein series}


\author{Ahmad El-Guindy}
\address{Ahmad El-Guindy, Department of Mathematics, Faculty of Science \\
Cairo University, Giza 12613, Egypt}
\curraddr{Ahmad El-Guindy, Texas A\&M University at Qatar, Science Program \\
Doha~23874, Qatar.}
\email{a.elguindy@gmail.com}
\thanks{}

\author{Aleksandar Petrov}
\address{}
\curraddr{Aleksandar Petrov, Texas A\&M University at Qatar, Science Program\\
 Doha~23874, Qatar.}
\email{aleksandar.petrov@qatar.tamu.edu}
\thanks{}



\date{}

\dedicatory{}

\commby{Matthew A. Papanikolas}

\begin{abstract}
We prove the equality of several $\tau$-recurrent sequences, which were first considered by Pellarin, and which have close connections to Drinfeld vectorial modular forms. Our result has several consequences: an $A$-expansion for the $l^\text{th}$~power ($1 \leq l \leq q$) of the deformation of the weight $2$ Eisenstein series; relations between Drinfeld modular forms with $A$-expansions; a new proof of relations between special values of Pellarin $L$-series. 
\end{abstract}

\maketitle

\tableofcontents

\section{Introduction and Statement of Results} \label{first}

Let $q=p^e$, with $p$ a prime number and $e$ a positive integer. Let $A$ be the polynomial ring $\Fq[\theta]$, $K$ its fraction field and $A_+$ the set of monic polynomials in~$A$. Let $| \cdot |$ be the absolute value on~$K$ uniquely defined by $|a| = q^{\deg_\theta (a)}$ for $a \in A$, and let $K_\infty$ be the completion of~$K$ with respect to $| \cdot |$, $\CCc$ be the completion of a fixed algebraic closure of~$K_\infty$. Let $B_r\subset \CCc$ be the open disc of radius $r$ centered at $0$. The \emph{Drinfeld upper half-plane} $\Omega$ is the set $\CCc \backslash K_\infty$ together with its rigid analytic structure as in \cite[\S~1.6]{Gos_mod_forms}. The group~$\Gamma := \GL_2 (A)$ acts on $\Omega$ by fractional linear transformations. 
Let $\phi_{\Car}$ be the  \emph{Carlitz module} defined by $\phi_\Car (\theta) = \theta \tau^0 + \tau$ with $\tau$ the $q$th power Frobenius operator on $\CCc$. We fix $\carperiod \in \CCc$, so that the lattice corresponding to the Carlitz module is $\carperiod A$. The exponential function of $\carperiod A$ will be called \emph{the Carlitz exponential} and will be denoted by $e_{\carperiod A}$.

Let $t$ be a new variable independent of $\theta$ and consider the series
\[
	\scar (t) := \sum_{n = 0}^\infty e_{\carperiod A} \left ( \frac{\carperiod}{\theta^{n+1}} \right) t^n.
\]
The series converges for $|t| < q$ (see \cite[Proposition~2.3]{BoPe}). 

The set of isomorphism classes of rank $2$ Drinfeld $A$-modules over $\CCc$ corresponds to $\Gamma  \backslash \Omega$ via the well-known equivalence of categories between Drinfeld modules and lattices (\cite[2.4]{Tha}). For $z \in \Omega$, let $\Lambda_z := z A \oplus A$. The exponential function for $\Lambda_z$ will be written as
\[
	e_{\Lambda_z} (\zeta) = \sum_{n = 0}^\infty \alpha_n (z) \zeta^{q^n},
\]
where $\alpha_n : \Omega \to \CCc$ are functions given explicitly in \cite[Theorem~3.1]{EP}. Following Pellarin \cite{Pel_order_of_vanishing} we consider 
\[
	\ss{1} (z, t) := \sum_{n = 0}^\infty e_{\Lambda_z} \left ( \frac{z}{\theta^{n+1}} \right) t^n, \qquad \qquad \ss{2} (z, t) := \sum_{n = 0}^\infty e_{\Lambda_z} \left (\frac{1}{\theta^{n+1}} \right) t^n.
\]
Both $s_1, s_2$ converge for $(z, t) \in \Omega \times B_q$. For~arithmetic consideration it is more convenient to work with normalizations of $\ss{1}$ and $\ss{2}$, namely
\[
	d_1 (z, t) := \carperiod \scar^{-1}(t) \ss{1}(z, t) ; \qquad \qquad d_2 (z, t) := \carperiod \scar^{-1}(t) \ss{2}(z, t) .
\]
The functions $d_1, d_2$ converge for any $z, t \in \CCc$  (see \cite[Proposition~19]{Pel_tau_recur_seq}).

Let $\chi_t : A \to \Fq[t]$ be the ring homomorphism defined by $\chi_t (a) = a(t)$. If $\alpha, \beta$ are positive integers, then the \emph{Pellarin $L$-function} is defined by 
\begin{equation} \label{L_definition}
	L(\chi_t^\alpha, \beta) := \sum_{a \in A_+} \chi_t(a)^\alpha a^{-\beta}.
\end{equation}

Pellarin introduced $L(\chi_t^\alpha, \beta)$ in \cite{Pel_L_series} as a deformation of the Carlitz zeta function and more general Goss $L$-functions (see \cite[Chapter~8]{Gos_book}).  In addition to Pellarin's original paper, the reader can find information about analytic continuation of Pellarin's $L$-function in \cite{Gos_Pel_L} (note that we will only consider values of $L(\chi_t^\alpha, \beta)$ for positive integers $\alpha, \beta$, i.e., values as in Equation \eqref{L_definition}), and formulas for special values in \cite{Perk}. In the course of the proof of our main result we will give a new proof of several relations between special values of Pellarin $L$-functions (see Corollary~\ref{Lvals}).

Let $\tau : \CCc ((t)) \to \CCc ((t))$ be the field automorphism that fixes $t$ and acts as  the Frobenius $q$th power operator on elements of $\CCc$. This agrees with the previous definition of $\tau$ on $\CCc$, so by abuse of notation we will use $\tau$ to denote both. If $f \in \CCc((t))$, then we will use the notation $f^{(i)}$ for $\tau^i f$. 

For $l \in \NN$ we define the sequence $\{\GG_{l, k} \}_{k \in \ZZ}$ by
\[
\GG_{l,k}:= \GG_{l, k} (z, t) = \frac{1}{L(\chi_t^l, l q^k)} \sideset{}{'} \sum_{c, d \in A} \left( \frac{ \over{c} d_1 + \over{d} d_2}{(c z + d)^{q^k}} \right)^l.
\]
The primed sum $\sideset{}{'} \sum$ will be used throughout to denote a sum with the term  where all summation indices are zero is omitted. We will prove (Proposition~\ref{convergence_of_E}) that if $k \geq 0$ the series defining $\GG_{l, k}$ is well-defined for all $(z, t) \in \Omega \times B_{q^{q^k}}$. 

Pellarin explicitly computed $\{\GG_{1, k} \}_{k \geq 0}$ in \cite[Theorem~4]{Pel_L_series}:
\begin{equation} \label{G_formula}
	\GG_{1, k} = -h^{q^k} (t - \theta^{q^k}) \scar^{(k)} \left ( d_2^{(k+1)} d_1 - d_1^{(k+1)} d_2 \right),
\end{equation}
where $h$ is the Drinfeld modular form of weight $q+1$ and type $1$, which is defined by Equation \eqref{define_h} below. We give a different formula for $\GG_{1, k}$ in \eqref{G_formula_two}.

The main result of the current paper is the computation of the sequence $\{\GG_{l, k} \}_{k \geq 0}$ for $l$ in the range $1 \leq l \leq q$:

\begin{theorem} \label{mainthm} Let $1 \leq l \leq q$ be fixed. For $k \geq 0$, we have
\[
	\GG_{l, k} = (-1)^{l+1} \GG_{1, k}^l.
\]
\end{theorem}

Prasenjit Bhowmik (work in preparation) has obtained similar results by a different method. Indeed, he computes Rankin brackets of certain families of Drinfeld modular forms and then applies a density argument (\cite{Pel_personal_communication}).

The paper is organized as follows. In Section~2 we introduce the necessary background, in particular, vectorial Drinfeld modular forms, their deformations and $\tau$-recurrent sequences. Subsection~2.3 gives a new method for the computation of the coefficients of $d_2$ based on the theory of shadowed partitions. In Section~3 we prove Theorem~\ref{mainthm}. Finally, Section~4 gives applications of Theorem~\ref{mainthm} to deformations of Drinfeld modular forms. In particular, Theorem~\ref{E^2} gives an $A$-expansion for the $l$th power of the deformation of Gekeler's `false Eisenstein series' (see \eqref{EE_definition}), $\EE^l$, $1 \leq l \leq q$, extending the one given by Pellarin for $l = 1$, while Corollary~\ref{application_to_A_exp} gives examples of Drinfeld modular forms that are eigenforms and can be expressed as products of eigenforms.

{\bf Acknowledgements} 
We are very grateful to Federico Pellarin for his encouragement and help, as well as for his many contributions to the field, including \cite{Pel_tau_recur_seq} which was the main inspiration for this work.
We would also like to thank David Goss,  Matthew Papanikolas, Rudolph Perkins, Dinesh Thakur and the anonymous referee for their feedback and support during the writing of the present article.

\section{Deformations of Vectorial Modular Forms}

\subsection{Drinfeld Modular Forms and Their Generalizations}

For $z \in \Omega$ and 
$\gamma = \begin{bmatrix} a & b \\
c & d \end{bmatrix} \in \Gamma$, we shall write $J_\gamma := c z + d$ and $L_\gamma := c/(cz + d)$.

The `imaginary distance' $|z|_i$ of $z \in \Omega$ is defined by $|z|_i:= \text{inf}_{x \in K_\infty} |z - x|$. Let $u := u(z) = 1/ e_{\carperiod A} (\carperiod z)$ be the normalized uniformizer at `infinity'.  We shall say that a rigid-analytic function $f : \Omega \to \CCc$ has a $u$-expansion, if there exists $\delta_f > 0$, such that for $z \in \Omega$ with $|z|_i > \delta_f$ we have \[
 	f (z) = \sum_{n = n_0}^\infty a_n u^n,
\]
for some $n_0 \in \ZZ$, $a_n \in \CCc$. Since $u \in \CCc$, $\tau$ acts on $u$ as a $q$th power Frobenius. The function $f$ is said to have an integral $u$-expansion if $n_0 \in \ZZ_{\geq 0}$. A~rigid-analytic function $f$ which satisfies $f(z + a) = f(z)$ for all $a \in A$ has a $u$-expansion and this $u$-expansion determines $f$ uniquely. 

For $c \in A_+$, set $u_c := u(cz)$. We have (see \cite[(6.2)]{Gek_88_paper} for instance)
\begin{equation} \label{uc}
	u_c = u^{q^{\deg_\theta(c)}} + \text{higher order terms in $u$}.
\end{equation}

\begin{definition}
A \emph{Drinfeld modular form} of weight $w$, type $m$ for $\Gamma$ is a rigid-analytic function $f : \Omega \to \CCc$ such that
\[
	f (\gamma (z)) = J_\gamma^w \det(\gamma)^{-m} f(z),
\]
and such that $f$ has an integral $u$-expansion. The set of Drinfeld modular forms of weight $w$ and type $m$ is a finite-dimensional $\CCc$-vector space, which we denote by~$\DMF{w}{m}$.
\end{definition}

Among the most important examples of Drinfeld modular forms are 
\begin{equation} \label{define_g}
	g  := 1 - (\theta^q - \theta) \sum_{c \in A_+} u_c^{q-1} = 1 - (\theta^q - \theta) u^{q-1} + \cdots \in \DMF{q-1}{0}, 
\end{equation}
\begin{equation} \label{define_h}
	h  := \sum_{c \in A_+} c^q u_c = u + \cdots \in \DMF{q+1}{1},  \qquad \del  := -h^{q-1} \in \DMF{q^2 -1}{0}.
\end{equation}
The forms $g$ and $h$ generate the space of Drinfeld modular forms of any weight and type. An important rigid-analytic function, which is not a Drinfeld modular form, but is closely connected with the theory, is Gekeler's `false Eisenstein series':
\[
	E  := \sum_{c \in A_+} c u_c.
\]
The reader can find more about the properties of $g, h, \del$ and $E$ in one of the standard references \cite{Gek_88_paper}, \cite{Gos_eisenstein}, \cite{Gos_mod_forms}.

Bosser and Pellarin \cite{BoPe} introduced the concept of \emph{almost-$A$-quasi-modular forms}, which encompasses Drinfeld modular forms as well as $d_2$ and the function 
\begin{equation} \label{EE_definition}
	\EE := - h d_2^{(1)}.
\end{equation} We do not recall the general definition of almost-$A$-quasi-modular forms (\cite[Definition~2.9]{BoPe}) here. 

 The function $d_1$ does not fall into the framework of almost-$A$-quasi-modular forms, but the following facts show that it has to be studied together with $d_2$ when considering modular properties.

\begin{proposition} \label{properties_of_d1_d2} 
\
\begin{enumerate}  \item The functions $d_1$, $d_2$ satisfy 
 \begin{equation} \label{mod_properties_d1_d2} J_\gamma d_1 (\gamma(z)) = \over{a} d_1  + \over{b} d_2, \qquad \qquad J_\gamma d_2 (\gamma (z)) = \over{c} d_1  + \over{d} d_2 . \end{equation} 
\item The function $d_2$ has a $u$-expansion with $\Fq[\theta, t]$ coefficients:
\begin{equation} \label{d2_expansion}
d_2 = 1 + (\theta - t) u^{q-1} + ( \theta - t) u^{(q-1)(q^2 - q + 1)} + \cdots,
\end{equation}
 while $d_1$ does not have a $u$-expansion.
 \item The functions $d_1$ and $d_2$ form a basis for the solution space of the $\tau$-difference equation
 \begin{equation} \label{tau_eq_ds}
 	X^{(2)} = \frac{1}{\del (t-\theta^q)} \left ( X - g X^{(1)} \right).
\end{equation}
\end{enumerate}
\end{proposition}

\begin{proof}
Results \eqref{mod_properties_d1_d2}, \eqref{d2_expansion}, are Lemmas~6, 8 in \cite{Pel_order_of_vanishing}, respectively. While \eqref{tau_eq_ds} is   \cite[(24)]{Pel_order_of_vanishing}.
\end{proof}

 Keeping with Pellarin's notation, let $\Tate_{<r}$ be the Tate algebra of formal power series $\sum_{n \geq 0} c_n t^n \in \CCc[[t]]$ that converge for $|t|<r$. If $r = \infty$, then we write $\Tate_\infty$ instead of $\Tate_{< \infty}$. Let $\mathcal{R}_{<r}$ be the ring that consists of formal series $\sum_{n \geq 0} f_n t^n$ such that: (1) for all $n$, $f_n$ is a rigid-analytic function $\Omega \to \CCc$; (2) for all $z \in \Omega$, $\sum_{n \geq 0} f_n (z) t^n$ is an element of $\Tate_{<r}$. 
The fraction field of $\mathcal{R} := \mathcal{R}_{<1}$ 
 will be denoted by $\mathcal{L}$, while the fraction field of 
\[
	\mathcal{R}_\infty = \bigcap_{r > 0} \mathcal{R}_{<r},
\]
will be denoted by $\mathcal{L}_\infty$. In this notation,  $d_1, d_2, \EE \in \mathcal{R}_\infty$.

\subsection{Vectorial Modular Forms and $\tau$-recurrent Sequences}

Let $\rho$ be a representation
\[
	\rho: \Gamma \to \GL_s ((\Fq((t))).
\]
The following definition is due to Pellarin.

\begin{definition}
A \emph{deformation of  a vectorial modular form} of weight $w$, dimension~$s$, type $m$ and radius $r > 0$ associated with the representation $\rho$ is a vector $\FF$ with entries in $\RadR_{<r}$, such that
\[
	\FF(\gamma(z)) = J_\gamma^w \det(\gamma)^{-m} \rho(\gamma) \FF(z), \qquad \forall \gamma \in \Gamma.
\]
The set of such vectors is denoted by $\DVMF{w}{m}{s}{\rho}{r}$. 
\end{definition}

Consider the representation $\rho_{t, 1}$ defined by
\[
	\rho_{t, 1} = \begin{bmatrix} \over{a} & \over{b} \\
	\over{c} & \over{d} \end{bmatrix},
\]
and its $l^\text{th}$ symmetric power $\rho_{t, l} := \Sym^l (\rho_{t, 1})$.  By definition $\rho_{t, l} = \Sym^{l} (\rho_{t, 1})$ can be realized on the vector space of homogeneous polynomials of degree $l$ via \[ X^i Y^{l-i} \mapsto (\over{a} X + \over{b} Y)^i (\over{c} X + \over{d} Y)^{l-i},\]
Let
\[
	 \Phi_l := {}^{\texttt{TR}} (d_1^l, d_1^{l-1} d_2, \ldots, d_1 d_2^{l-1}, d_2^l ), \]  where $\mathtt{TR}$ is the usual transpose. Property \eqref{mod_properties_d1_d2} implies $\Phi_l \in \DVMF{-1}{0}{l+1}{\rho_{t, l}}{ \infty}$.

We let $\Evec_l$ be defined to be the transpose of the row vector
\[
	 \frac{1}{L(\chi_t^l, l)} \sideset{}{'} \sum_{c, d \in A} \left  (\frac{\over{c}^l}{(cz + d)^l}, \frac{\binom{l}{1} \over{c}^{l-1} \over{d}}{(cz + d)^l}, \ldots, \frac{\binom{l}{l-1} \over{c} \over{d}^{l-1}}{(cz + d)^l}, \frac{\over{d}^l}{(cz + d)^l} \right ).
\] 
 An equivalent (but more `modular') definition\footnote{The reader should be aware that in his original preprint Pellarin omits the binomial coefficients from the definition of $\Evec_l$, but we have confirmed with Pellarin that the binomial coefficients in the definition of $\Evec_l$ need to be present.} and other properties of $\Evec_l$ can be found in \cite[Section~3.3.2]{Pel_tau_recur_seq}. Part $1$ from \cite[Proposition~21]{Pel_tau_recur_seq} shows that for $\gamma \in \Gamma$, we have
 \begin{equation} \label{transformation_property_of_E}
 \Evec_l (\gamma (z)) = J_\gamma^l \left ( {}^{\mathtt{TR}} (\rho_{t, l}^{-1}) (\gamma)  \right) \Evec_l (z).
 \end{equation}

\begin{proposition} \label{convergence_of_E}
The series defining $ \tau^k \Evec_l$, $k \geq 0$,  converges for $(z, t) \in \Omega \times B_{q^{q^k}}$. 
\end{proposition}

\begin{proof} Note that if $f(t)$ converges for $|t| < r$, then $\tau f(t)$ converges for $|t| <  r^q$, thus it suffices to prove the case $k=0$. Write $|t| = q^\epsilon$, with $\epsilon < 1$. For such $t$, the series
\[
	L(\chi_t^l, l) = \sum_{c \in A_+} \frac{\over{c}^l}{c^l}
\]
converges, since
\[
	\lim_{|c| \to \infty} \left | \over{c}^l c^{-l} \right | = \lim_{|c| \to \infty} |c|^{(\epsilon-1) l} = 0.
\]
Assume that $|z| \geq 1$. By property \eqref{transformation_property_of_E}, we see that, for $a \in A$,  $\Evec_l (z + a) = M \Evec_l (z)$, where $M$ is a matrix with coefficients in $\mathbb{C}_\infty[t]$ that do not depend on $z$. Therefore the convergence of the series defining $\Evec_l$ is not affected by transformations of the form $z \mapsto z + a$, $a \in A$. Since $K_\infty$ is locally compact we know that $|z|_i = |z - x_0|$ for some $x_0 \in K_\infty$. By applying transformations of the form $z \mapsto z + a$, $a \in A$, we can assume $|x_0| < 1$, hence $|z|_i = |z - x_0| = \text{max}\{  |z|, |x_0| \} = |z|$. We can therefore assume without loss of generality that $|z| = |z|_i$. For such $z$, we have $|c z + d| = \text{max} \{ |cz|, |d| \}$ for $c, d \in A$.

We need to show that the series
\begin{equation} \label{series_star}
	\sideset{}{'} \sum_{c, d \in A} \frac{\over{c}^i \over{d}^{l-i}}{(cz + d)^l}
\end{equation}
converges for any $i$ such that $0 \leq i \leq l$. 

If $|cz| < |d|$, then
\[
     \left |	\frac{\over{c}^i \over{d}^{l-i}}{ (cz + d)^l } \right | = \frac{|c|^{i \epsilon} |d|^{(l-i) \epsilon}}{|d|^l} < \frac{1}{|z|^{\epsilon i}} |d|^{(\epsilon - 1) l}.
\]
 Since $|cz| < |d|$ implies that $|c| \to \infty \Rightarrow |d| \to \infty$, the last quantity tends to $0$ as $|c| \to \infty$ or $|d| \to \infty$. A similar argument shows that when $|cz| > |d|$ if  $|c| \to \infty$ or $ |d| \to \infty$ we have
 \[
  \left |	\frac{\over{c}^i \over{d}^{l-i}}{ (cz + d)^l } \right | \to 0.
 \]
 We conclude that for $|z| \geq 1$, $|t| = q^\epsilon$, $\epsilon < 1$, the series \eqref{series_star} converges. Since every $z \in \Omega$ is equivalent under the action of $\Gamma$ to an element in $\mathfrak{F} = \{ z \in \Omega : |z|_i = |z| \geq 1 \}$ and $\Evec_l$ satisfies \eqref{transformation_property_of_E} under the action of $\Gamma$ (i.e., the action of $\Gamma$ permutes the components of $\Evec_l$), the result for $|z| \geq 1$ implies that for all $z \in \Omega$. \end{proof}

Property \eqref{transformation_property_of_E} and the previous proposition show that $\Evec_l \in \DVMF{l}{0}{l+1}{ ^{\mathtt{TR}}\rho_{t, l}^{-1}}{q}$. By definition \begin{equation} \GG_{l, k} = (\tau^k \Evec_l) \cdot  \Phi_l,\end{equation} where the dot denotes the usual inner product of vectors.  The modular properties of $\Evec_l$ and $\Phi_l$ imply that for $k \geq 0$, \[\GG_{l, k} \in \DVMF{lq^k-l}{0}{1}{\mathbf{1}}{q^{q^k}},\] where $\mathbf{1}$ is the trivial representation.

Next we recall the theory of $\tau$-recurrent sequences as described in \cite[Section~2]{Pel_tau_recur_seq}. Let $\tauField$ be a field  together with an infinite order automorphism $\tau$. The fixed field of $\tau$ will be denoted by $\tauField^\tau$. Let $L = A_0 \tau^0 + \cdots + A_s \tau^s \in \tauField[\tau]$ be a $\tau$-linear operator such that $A_0 \neq 0, A_s \neq 0$ ($s$ is the \emph{order} of $L$). Given a sequence $\GG = \{ \GG_k \}_{k \in \ZZ}$ with elements in $\tauField$, we write $L(\GG)$ for the sequence
\[
	\{ A_0 \tau^0 \GG_k + \cdots + A_s \tau^s \GG_{k-s} \}_{k \in \ZZ}.
\]
The sequence $\GG$ is a \emph{$\tau$-recurrent sequence} for $L$ if $L(\GG) \equiv 0$.  The space of all $\tau$-recurrent sequences for $L$ is a finite-dimensional $\tauField$-vector space, which we denote by~$V(L)$. The space of solutions to the associated \emph{$\tau$-difference equation}
\begin{equation} \label{general_tau_dif_eq}
	A_0 \tau^0 X + \cdots + A_s \tau^s X = 0
\end{equation}
is denoted by $V^\tau(L)$. Any solution $x$ to \eqref{general_tau_dif_eq} gives an element of $V(L)$ by simply taking the constant sequence $\{ x \}_{k \in \ZZ}$. Pellarin shows (see \cite[Propositions~10, 11]{Pel_tau_recur_seq}) that if  $x_1, \ldots, x_s$ are $\tauField^\tau$-linearly independent elements of $\tauField$, then there exists an explicit procedure for computing a  $\tau$-linear operator $L$ of order $s$ (unique if we assume the normalization $A_s=1$; else $L$ is unique up to left multiplication) such that  $\{x_i:1\leq i\leq s\}$ is a basis of $V^\tau(L)$. In addition, for any vector $\Evec = (e_1, \ldots, e_s) \in \tauField^s$ the sequence $\GG$, defined by
\[
	 \GG_{k} = (\tau^k \Evec) \cdot (x_1, \ldots, x_s),
\]
belongs to $V(L)$ and any $\GG \in V(L)$ is of this form for a unique $\Evec \in \tauField^s$.

Applying this to $\tauField = \mathcal{L}$, $(x_1, \ldots, x_s) = (d_1^l, d_1^{l-1} d_2, \ldots, d_2^{l})$ we have

\begin{proposition}
 The sequence $\{\GG_{l, k} \}_{k \in \ZZ}$, defined by
 \[
 	\GG_{l, k} := (\tau^k \Evec_l) \cdot  \Phi_l , \qquad \qquad k \in \ZZ, 
\]
 is a $\tau$-recurrent sequence that satisfies the unique normalized $\tau$-difference equation $L_l$ satisfied by   $d_1^l,  d_1^{l-1} d_2, \ldots, d_1 d_2^{l-1},  d_2^l$. 
 \end{proposition}
 
Note that the linear independence of $\{ d_1^l, d_1^{l-1} d_2, \ldots, d_1 d_2^{l-1}, d_2^l \}$ over $\tauField^\tau = \Fq(t)$ is also part of the result (see \cite[Lemmas~14 \& 20]{Pel_tau_recur_seq}). 
 
\noindent
{\bf Examples:}
The sequence $\{ \GG_{1, k} \}_{k \in \ZZ}$ is a $\tau$-recurrent sequence for \[ L_1 := \tau^0 - g \tau^1 - \del (t - \theta^q) \tau^2 \] and $\{ \GG_{2, k} \}_{k \in \ZZ}$ is a $\tau$-recurrent sequence for:
\begin{equation} \label{L_for_G2} \begin{aligned}
 & L_2 :=  \tau^0 - g^{1-q} (g^{1 + q} + \del (t - \theta^q)) \tau^1 \\
&  \qquad  \ \ - \del (t - \theta^q) (g^{1 + q} + \del (t-\theta^q)) \tau^2 + g^{1-q} \del^{1 +  2q} (t - \theta^q) (t - \theta^{q^2})^2 \tau^3.
\end{aligned}
\end{equation}

In \cite[Theorem~4]{Pel_L_series} Pellarin determines $\{ \GG_{1, k} \}_{k \in \ZZ}$ completely by computing its first two non-negative terms: $\GG_{1, 0} = -1$ and $\GG_{1, 1} = -g$. Our main theorem (Theorem~\ref{mainthm}) shows that $\GG_{l, k} = (-1)^{l+1} \GG_{1, k}^l$. 

It follows from the following proposition that $\GG_{1, k}^l$ is part of a basis for $V(L_l)$.

\begin{proposition}
Let $\{(\calG_k), (\calH_k)\}$ be a basis of $V(L_1)$, then for any $l \geq 1$,
\[
\{(\calG_k^i \calH_k^{l-i}): 0\leq i\leq l \} 
\]
forms a basis of $V(L_l)$, where $L_l$ is the unique normalized $\tau$-difference equation satisfied by   $d_1^l,  d_1^{l-1} d_2, \ldots, d_1 d_2^{l-1},  d_2^l$.
\end{proposition}
\begin{proof}
According to \cite[Proposition~11]{Pel_tau_recur_seq}, there exist $a,b,c,d \in \CCc[t]$ such that
\begin{equation}
\calG_k=a^{(k)}d_1+b^{(k)}d_2, \qquad \qquad  \calH_k=c^{(k)}d_1+d^{(k)}d_2.
\end{equation}
We thus get 
\begin{equation}\label{prodsol}
\calG^i_k \calH^{l-i}_k=\sum_{m=0}^l C_{mi}^{(k)} d_1^{m}d_2^{l-m},
\end{equation}
where
\[
C_{mi}=\sum_{j=0}^i \binom{i}{j}\binom{l-i}{m-j}a^j b^{i-j} c^{m-j}d^{l-m-(i-j)}.
\]
Thus, again by the same proposition from \cite{Pel_tau_recur_seq}, we see that \eqref{prodsol} is equivalent to $(\calG_k^i\calH_k^{l-i})\in V(L_l)$. The linear independence of $(\calG_k)$ and $(\calH_k)$ is equivalent to $ad-bc\neq 0$. However, we have the following equality of matrices
\[
(C_{mi})_{0\leq i, m \leq l}=\Sym^l\left(\begin{array}{ll}
a&b\\
c&d
\end{array}\right),
\]
which yields
\[
\det{(C_{mi})}=(ad-bc)^{\frac{l^2+l}{2}}\neq 0,
\]
again implying the linear independence of our proposed basis. This proves the result as $V(L_l)$ has dimension at most $l+1$.
\end{proof}

\subsection{Using Shadowed Partitions to Approximate $d_2$}

In this subsection, we give a new formula \eqref{G_formula_two} for the sequence $\GG_{1, k}$, $k \geq 1$. This formula can be used to give a method for computing the $u$-expansion of $d_2$ which, in contrast with the original computation of Pellarin, is not recursive in the coefficients of $d_2$. This is done via \emph{shadowed partitions} as in~\cite{EP2}. If $S \subset \ZZ$ and $j \in \ZZ$, then let $S + j := \{ i + j : i \in S \}$. Let $r, n \in \NN$. We define the \emph{order $r$ index-shadowed partition of $n$} by
\[
\begin{aligned}
	P_r (n) & :=  \big \{ (S_1, S_2, \ldots, S_r) : S_i \subset \{ 0, 1, \ldots, n-1 \},   \\
		& \text{and } \{ S_i + j : 1 \leq i \leq r, 0 \leq j \leq i-1 \} \text{ form a partition of } \{ 0, 1, \ldots, n-1 \}  \big \}.
\end{aligned}
\]

\begin{theorem} \label{d_2_approximation} For $k \geq 1$, we have
\begin{equation} 
	d_2 - \sum_{(S_1 , S_2) \in P_2 (k)} \prod_{j \in S_1} \prod_{i \in S_2} g^{q^{j}} (t - \theta^{q^{i + 1}}) \del^{q^i} \in u^{q^{k-1} (q-1)} \Fq[\theta, t][[u]].
\end{equation}
\end{theorem}

\begin{proof} According to Lemma~3.3 from \cite{EP2}, for $k \geq 1$ we have
\begin{equation} \label{G_formula_two}
\GG_{1, k} = - \sum_{(S_1, S_2) \in P_2 (k)} \prod_{j \in S_1} \prod_{i \in S_2} g^{q^{j}} (t - \theta^{q^{i+1}}) \del^{q^i}.
\end{equation}
Using this we see that $\GG_{1, 1} = -g$, $\GG_{1, 2} = -g^{q+1} - (t - \theta^q) \del$. By comparing the $u$-expansions of $d_2$ \eqref{d2_expansion} with the $u$-expansions of $\GG_{1, 1}$ and $\GG_{1, 2}$ we see that \[ d_2 + \GG_{1, 1} \in u^{q-1} \Fq[\theta, t][[u]] \quad \text{and} \quad  d_2 + \GG_{1, 2} \in u^{q(q-1)} \Fq[\theta, t][[u]]. \] Since~$\GG_{1, k}$ is a $\tau$-recurrent sequence for $L_1$ and $L_1(d_2) = 0$, it follows by induction that $d_2 + \GG_{1, k} \in u^{q^{k-1} (q-1)} \Fq[\theta, t][[u]]$.
\end{proof}


\section{The Proof of Theorem \ref{mainthm}} \label{the_proof}

We start the proof of Theorem \ref{mainthm} with several lemmas.

\begin{lemma} \label{lemma1}
\[
	1 + \sum_{u \in \Fq} \frac{X + u}{Y + u}  = \frac{Y^q - X}{Y^q - Y}.
\]
\end{lemma}

\begin{proof}   \renewcommand{\qedsymbol}{} It is well-known that $\prod_{u\in \Fq}(Y-u)=Y^q-Y$. By  logarithmic differentiation: 
\[
	\sum_{u \in \Fq} \frac{1}{Y + u} = \frac{-1}{Y^q - Y}.
\]
If $1 \leq l \leq q-1$, then
\[
	\sum_{u \in \Fq} \frac{Y^l}{Y - u} - \sum_{u \in \Fq} \frac{u^l}{Y - u}  = \sum_{u \in \Fq} \sum_{j = 0}^{l-1} Y^j u^{l-1-j} = \sum_{j = 0}^{l-1} Y^j \sum_{u \in \Fq} u^{l-1-j}  = 0.
\]
Hence, for $1 \leq l \leq q-1$,
\[
	\sum_{u \in \Fq} \frac{u^l}{Y + u} =  (-1)^l \sum_{u \in \Fq} \frac{u^l}{Y - u} = (-1)^l \sum_{u \in \Fq} \frac{Y^l}{Y + u} = \frac{ (-1)^{l+1} Y^l}{Y^q - Y}.
\]
Now the result is a simple computation
\[
	 \sum_{u \in \Fq} \frac{X + u}{Y + u} =   \sum_{u \in \Fq}  \frac{X}{Y + u} +  \sum_{u \in \Fq} \frac{u}{Y + u}  = - \frac{X}{Y^q - Y} + \frac{Y}{Y^q - Y} = \frac{Y^q - X}{Y^q - Y} - 1.~\square
\]
\end{proof}

Next, we recall Lucas' Theorem \cite{Lucas} which states that for prime $p$, positive $e$, and $0 \leq n_0, n_1, \dots, n_m, i_0, i_1, \dots, i_m \leq p^e-1$ we have
\[
\binom{n_0+n_1p^e+\dots n_m p^{em}}{i_0+i_1p^e+\dots i_m p^{em}}\equiv \binom{n_0}{i_0}\binom{n_1}{i_1}\cdots\binom{n_m}{i_m} \pmod p.
\] 
We will use Lucas' Theorem in the proof of the next lemma.

\begin{lemma} \label{lemma2}
For $1 \leq l \leq q$, 
\[
	1 + \sum_{u \in \Fq} \left ( \frac{X + u}{Y + u} \right)^l = \left ( 1 + \sum_{u \in \Fq} \left (  \frac{X + u}{Y + u} \right) \right)^l 
\]
\end{lemma}

\begin{proof} According to Lemma~\ref{lemma1}, we have
\[
	\left ( 1 + \sum_{u \in \Fq} \left (  \frac{X + u}{Y + u} \right) \right)^l  = \left ( \frac{Y^q - X}{Y^q - Y} \right)^l.
\]
Consider 
\[
	P(X) := 1 + \sum_{u \in \Fq} \left ( \frac{X + u}{Y + u} \right)^l
\]
as a polynomial in $X$. If $X = Y^q$, then
\[
\begin{aligned}
	P(Y^q) & = 1 + \sum_{u \in \Fq} \left ( \frac{Y^q + u}{Y + u} \right)^l = 1 + \sum_{u \in \Fq} (Y + u)^{l(q-1)} \\ \\
	& = 1 + \sum_{u \in \Fq} \sum_{i = 0}^{l(q-1)} \binom{l(q-1)}{i} Y^{l(q-1) - i} u^i  = 1 - 1 - \sum_{i = 1}^{l-1} \binom{l(q-1)}{i(q-1)} Y^{(l-i)(q-1)}\\
	& = 1 - 1 = 0.
\end{aligned}
\] 
We have used that for $1\leq i \leq l-1$, $\binom{l(q-1)}{i(q-1)} \equiv 0$ in $\Fq$, which follows by Lucas' Theorem since in $\Fq$ we have
\[
\binom{l(q-1)}{i(q-1)}=\binom{(l-1)q+(q-l)}{(i-1)q+(q-i)}=\binom{l-1}{i-1}\binom{q-l}{q-i}=0
\]
as $q-l<q-i$. Thus $X = Y^q$ is a root of $P(X)$. 

Next, we show that this root repeats $l$ times. To that end, we compute for any $1 \leq j \leq l-1$
\[
\begin{aligned}
	\frac{d^j P}{d X^j} (Y^q) & = l (l-1) \cdots (l- j+1) \sum_{u \in \Fq} (Y + u)^{(l-j)q - l} \\
	& = l (l-1) \cdots (l-j+1) \sum_{u \in \Fq} \sum_{i = 0}^{(l-j)q - l} \binom{(l-j)q - l}{i} Y^{(l-j)q - l - i} u^i \\
	& = - l (l-1) \cdots (l-j+1) \sum_{i = 1}^{l - \lceil \frac{jq}{q-1} \rceil} \binom{(l-j) q - l}{i (q-1)} Y^{(l-j) q - l - i(q-1)}  = 0, \\
\end{aligned}
\]
since all the binomial coefficients vanish by Lucas' Theorem.

This shows that $P(X)$ and $\left ( \frac{Y^q - X}{Y^q - Y} \right)^l$ are equal up to a constant. Substituting $X = Y$ shows that the constant is $1$ and finishes the proof.
\end{proof}

\begin{lemma} \label{lemma3}
Let $F$ be a field that contains $\Fq$. Let $V_1, \ldots, V_n$ be arbitrary elements of $F$ and $W_1, \ldots, W_n$ be a set of $\Fq$-linearly independent elements of $F$. For $1 \leq l \leq q$, 
\[
	\sideset{}{'} \sum_{u_1, \ldots, u_n \in \Fq} \left ( \frac{u_1 V_1 + \cdots + u_n V_n}{u_1 W_1 + \cdots + u_n W_n} \right)^l = (-1)^{l+1} \left ( \sideset{}{'} \sum_{u_1, \ldots, u_n \in \Fq} \frac{u_1 V_1 + \cdots + u_n V_n}{u_1 W_1 + \cdots + u_n W_n} \right)^l.
\]
\end{lemma}

\begin{proof}
We will use induction on $n$. If $n = 1$, then we have
\[
	\sideset{}{'} \sum_{u \in \Fq} \left ( \frac{u V_1}{u W_1} \right)^l = \left ( \frac{V_1}{W_1} \right )^l \sideset{}{'} \sum_{u \in \Fq} 1 =  - \left ( \frac{V_1}{W_1} \right)^l.
\]
While
\[
	(-1)^{l+1} \left ( \sideset{}{'} \sum_{u \in \Fq} \frac{u V_1}{u W_1} \right)^l = (-1)^{l+1} \left ( \frac{V_1}{W_1} \right)^l \left ( \sideset{}{'} \sum_{u \in \Fq} 1 \right)^l  = - \left ( \frac{V_1}{W_1} \right)^l.
\]

Assume that the result holds for $n-1$. Until the end of the proof, let $u_1, \ldots, u_n$ be elements of $\Fq$. 

First, assume that $V_n \neq 0$. Set $V_i' := \frac{V_i}{V_n}, W_i' := \frac{W_i}{W_n}$. We have
\[
\begin{aligned}
	&\sideset{}{'} \sum_{u_1, \ldots, u_n} \left ( \frac{u_1 V_1 + \cdots + u_n V_n}{u_1 W_1 + \cdots + u_n W_n} \right)^l  \\
	& =  \sideset{}{'} \sum_{u_1, \ldots, u_{n-1}} \sum_{u_n} \left ( \frac{u_1 V_1 + \cdots + u_n V_n}{u_1 W_1 + \cdots + u_n W_n} \right)^l + \sideset{}{'} \sum_{u_n} \left ( \frac{u_n V_n}{u_n W_n} \right)^l \\
	& = \left( \frac{V_n}{W_n} \right)^l \left (  -1 + \sideset{}{'} \sum_{u_1, \ldots, u_{n-1}} \sum_{u_n} \left ( \frac{u_1 V_1' + \cdots + u_n}{u_1 W_1' + \cdots + u_n} \right)^l  \right), & \\
	& = \left( \frac{V_n}{W_n} \right)^l  \sideset{}{'} \sum_{u_1, \ldots, u_{n-1}}  \left ( \sum_{u_n} \left ( \frac{u_1 V_1' + \cdots + u_n}{u_1 W_1' + \cdots + u_n} \right)^l  + 1 \right)  \\
	& = \left( \frac{V_n}{W_n} \right)^l  \sideset{}{'} \sum_{u_1, \ldots, u_{n-1}}  \left ( \sum_{u_n} \left ( \frac{u_1 V_1' + \cdots + u_n}{u_1 W_1' + \cdots + u_n} \right) + 1 \right)^l ,  & \text{by Lemma~\ref{lemma2}}\\
	& = \left( \frac{V_n}{W_n} \right)^l \sideset{}{'} \sum_{u_1, \ldots, u_{n-1}} \left ( \frac{u_1 V_1'' + \cdots + u_{n-1} V_{n-1}''}{u_1 W_1'' + \cdots + u_{n-1} W_{n-1}''} \right)^l, & \text{by Lemma~\ref{lemma1}} \\
\end{aligned}
\]
where $V_i'' = (W_i')^q - V_i', W_i'' = (W_i')^q - W_i'$. Note that the linear independence over $\Fq$ of $W_i''$ for $1\leq i\leq n-1$, is equivalent to the linear independence over $\Fq$ of $W_j$ for $1\leq j\leq n$. By the induction hypothesis the last expression is equal to
\[
	\begin{aligned}
	 (-1)^{l+1} \left( \frac{V_n}{W_n} \right)^l  \left ( \sideset{}{'} \sum_{u_1, \ldots, u_{n-1}} \frac{u_1 V_1'' + \cdots + u_{n-1} V_{n-1}''}{u_1 W_1'' + \cdots + u_{n-1} W_{n-1}''} \right)^l .
	\end{aligned}
\]
On the other hand, 
\[
	\begin{aligned}
	\left ( \sideset{}{'} \sum_{u_1, \ldots, u_n} \frac{u_1 V_1 + \cdots + u_n V_n}{u_1 W_1 + \cdots + u_n W_n} \right)^l & = \left (\frac{V_n}{W_n} \right)^l \left ( \sideset{}{'} \sum_{u_1, \ldots, u_n} \frac{u_1 V_1' + \cdots + u_n }{u_1 W_1' + \cdots + u_n} \right)^l \\
	& =  \left( \frac{V_n}{W_n} \right)^l  \left ( \sideset{}{'} \sum_{u_1, \ldots, u_{n-1}} \frac{u_1 V_1'' + \cdots + u_{n-1} V_{n-1}''}{u_1 W_1'' + \cdots + u_{n-1} W_{n-1}''} \right)^l ,
	\end{aligned}
\]
where the last equality follows from Lemma~\ref{lemma1}.

Finally, assume that $V_n = 0$. We compute
\[
\begin{aligned}
	\sideset{}{'} \sum_{u_1, \ldots, u_n} \left ( \frac{u_1 V_1 + \cdots + u_n V_n}{u_1 W_1 + \cdots + u_n W_n} \right)^l  &  =  \sideset{}{'} \sum_{u_1, \ldots, u_{n-1}} \sum_{u_n} \left ( \frac{u_1 V_1 + \cdots + u_{n-1} V_{n-1}}{u_1 W_1 + \cdots + u_{n-1} W_{n-1} + u_n W_n} \right)^l \\
	& = W_n^{-l}   \sideset{}{'} \sum_{u_1, \ldots, u_{n-1}} \sum_{u_n} \left ( \frac{u_1 V_1 + \cdots + u_{n-1} V_{n-1}}{u_1 W_1' + \cdots + u_{n-1} W_{n-1}' + u_n} \right)^l \\
	& = W_n^{-l}  \sideset{}{'} \sum_{u_1, \ldots, u_{n-1}}  \left ( \sum_{u_n}  \frac{u_1 V_1 + \cdots + u_{n-1} V_{n-1}}{u_1 W_1' + \cdots + u_{n-1} W_{n-1}' + u_n} \right)^l \\
	& = (-1)^l W_n^{-l}  \sideset{}{'} \sum_{u_1, \ldots, u_{n-1}} \left( \frac{u_1 V_1 + \cdots + u_{n-1} V_{n-1}}{u_1 W_1'' + \cdots + u_{n-1} W_{n-1}'' } \right)^l
\end{aligned}
\]
The second to last equality is the equality of Goss polynomials
$
	\sum_{u \in \Fq} \left ( \frac{1}{Y + u} \right)^l = \left ( \sum_{u \in \Fq} \frac{1}{Y + u} \right)^l.
$
On the other hand,
\[
\begin{aligned}
	& (-1)^{l+1}	\left ( \sideset{}{'} \sum_{u_1, \ldots, u_n} \frac{u_1 V_1 + \cdots + u_{n-1} V_{n-1}}{u_1 W_1 + \cdots + u_{n-1} W_{n-1} + u_n W_n} \right)^l  \\
	& = (-1)^{l+1} W_n^{-l} \left ( \sideset{}{'} \sum_{u_1, \ldots, u_{n-1}} \sum_{u_n} \frac{u_1 V_1 + \cdots + u_{n-1} V_{n-1}}{u_1 W_1' + \cdots + u_{n-1} W_{n-1}' + u_n} \right)^l  \\
	& = (-1)^{2l+1} W_n^{-l} \left ( \sideset{}{'} \sum_{u_1, \ldots, u_{n-1}} \frac{u_1 V_1 + \cdots + u_{n-1} V_{n-1}}{u_1 W_1'' + \cdots + u_{n-1} W_{n-1}''} \right)^l  .
\end{aligned}
\]
Thus the result for $n$ follows from the result for $n-1$, completing the proof.
\end{proof}

Lemma \ref{lemma3} gives a new proof of the following relations between some values of Pellarin $L$-functions, which are special cases of Theorem~1.3 in \cite{Perk}.

\begin{corollary} \label{Lvals}
Let $1 \leq l \leq q$. We have
\[
	L(\chi_t^l, l) = L(\chi_t, 1)^l.
\]
\end{corollary}

\begin{proof}
Let $A(n) = \{ a=\sum_{i=0}^{n-1}a_i\theta^i$: $a_i \in \Fq \}$. Applying Lemma \ref{lemma3} with $W_i=\theta^{i-1}$ and $V_i=t^{i-1}$ we see that
\[
\sideset{}{'} \sum_{a \in A(n)} \frac{\over{a}^l}{a^l} =  \left ( \sideset{}{'} \sum_{a \in A(n)} \frac{\over{a}}{a} \right)^l.
\]
But we have
\[
	- L(\chi_t^l, l) = \sideset{}{'} \sum_{a \in A} \frac{\over{a}^l}{a^l}  = \lim_{n \to \infty} \sideset{}{'} \sum_{a \in A(n)} \frac{\over{a}^l}{a^l},
\]
and a short calculation gives the result.
\end{proof}

We are ready to complete the
{proof of Theorem~\ref{mainthm}:} 
\begin{proof}
As in the proof of Theorem~\ref{Lvals}, let $A(n) = \{ a \in A: \deg (a) < n \}$. Applying Lemma~\ref{lemma3} with 
\[
V_i=\begin{cases}
d_1t^{i-1} & \textrm{ if } 1\leq i \leq n,\\
d_2t^{i-n-1}& \textrm{ if } n+1\leq i\leq 2n, 
\end{cases}
\]
and 
\[
W_i=\begin{cases}
z\theta^{i-1} & \textrm{ if } 1\leq i \leq n,\\
\theta^{i-n-1}& \textrm{ if } n+1\leq i\leq 2n, 
\end{cases}
\]
we have
\begin{equation} \label{star_star}
	\sideset{}{'} \sum_{c, d \in A(n)} \left( \frac{ \over{c} d_1 + \over{d} d_2}{(c z + d)^{q^k}} \right)^l = (-1)^{l
	+1} \left ( \sideset{}{'} \sum_{c, d \in A(n)} \frac{\over{c} d_1 + \over{d} d_2}{(c z + d)^{q^k}} \right)^l.
\end{equation}
But
\[
	\GG_{l, k} = \frac{1}{\tau^k L(\chi_t^l, l)}  \lim_{n \to \infty}  \sideset{}{'} \sum_{c, d \in A(n)} \left( \frac{ \over{c} d_1 + \over{d} d_2}{(c z + d)^{q^k}} \right)^l.
\]
By Theorem \ref{Lvals} we have $L(\chi_t^l, l) = L(\chi_t, 1)^l,$
combining this with \eqref{star_star}, we see that
\[
	\GG_{l, k} = (-1)^{l + 1} \GG_{1, k}^l.
\]
\end{proof}

\section{Consequences of Theorem~\ref{mainthm}} \label{sec5}

\begin{theorem} \label{sum_c_d_to_product}
Let $1 \leq l \leq q$, $0 \leq j \leq l$. For $(z, t) \in \Omega \times B_q$ we have
\[
	\sum_{c \in A_+} \sum_{d \in A} \frac{ \over{c}^{l-j} \over{d}^j}{(cz + d)^l} = \left ( \sum_{c \in A_+} \sum_{d \in A} \frac{\over{c}}{cz + d} \right)^{l-j} \left ( \sum_{c \in A_+} \sum_{d \in A} \frac{\over{d}}{cz + d} \right)^j.
\]
\end{theorem}

\begin{proof} By Theorem~\ref{mainthm}, if $1 \leq l \leq q$, then 
\[
	\GG_{l, k} = (-1)^{l+1} \GG_{1, k}^l, \qquad \qquad \forall k \in \ZZ.
\]
Since $\GG_{l, k} = (\tau^k \Evec) \cdot  \Phi_l$ for a unique vector $\Evec$ (see \cite[Proposition~10]{Pel_tau_recur_seq}) it follows that
\[
	\sideset{}{'} \sum_{c, d \in A} \frac{\over{c}^{l-j} \over{d}^j}{(cz + d)^l} = (-1)^{l+1}  \left ( \sideset{}{'} \sum_{c, d \in A} \frac{\over{c}}{cz + d} \right)^{l-j} \left ( \sideset{}{'} \sum_{c, d \in A} \frac{\over{d}}{cz + d} \right)^j.
\]
The proof is complete by observing that for $0 \leq j \leq l$:
\[
	\sideset{}{'} \sum_{c, d \in A} \frac{\over{c}^{l-j} \over{d}^j}{(cz + d)^l} = - \sum_{c \in A_+} \sum_{d \in A}  \frac{\over{c}^{l-j} \over{d}^j}{(cz + d)^l}.
\]
\end{proof}

\begin{remark}
Since $\GG_{l, k}$ is completely determined by $\Evec_l$ one sees that Theorem~\ref{sum_c_d_to_product}  is in fact equivalent to Theorem~\ref{mainthm}. 
\end{remark}

The result of Theorem~\ref{sum_c_d_to_product} when $j = 0$ actually holds for $|t| < q^q$ as long as $z$ is in the neighborhood of `infinity' $\Omega_1 = \{ z \in \Omega: |z|_i > 1 \}$.

\begin{corollary} \label{improved_convergence} Let $1 \leq l \leq q$. For $(z, t) \in \Omega_1 \times B_{q^q}$, we have
\[
	\sum_{c \in A_+} \sum_{d \in A} \frac{\over{c}^l}{(cz + d)^l} = \left ( \sum_{c \in A_+} \sum_{d \in A} \frac{\over{c}}{(cz + d)} \right ) ^l.
\] 
\end{corollary}

\begin{proof} According to \cite[(5.5)]{Gek_88_paper} and \eqref{uc} for $z \in \Omega_1$
\[
	| \carperiod u(cz)| = \left | \sum_{d \in A} \frac{1}{cz + d} \right| \leq q^{-|c|}.
\]
If, in addition, $|t| < q^q$, then $|\chi_t (c)| < q^{|c|}$ and therefore the series 
\[
\carperiod \sum_{c \in A_+} \over{c} u(c z) = \sum_{c \in A_+} \sum_{d \in A} \frac{\over{c}}{(cz + d)}
\]
converges. The same estimates and properties of Goss polynomials show that for $(z, t) \in \Omega_1 \times B_{q^q}$ the series
\[
	\carperiod^l  \sum_{c \in A_+} \over{c}^l u (c z)^l = \sum_{c \in A_+} \sum_{d \in A} \frac{\over{c}^l}{(cz + d)^l}
\]
converges. By Theorem \ref{sum_c_d_to_product} we know that for $|t| < q$, $1 \leq l \leq q$,
\[
	\sum_{c \in A_+} \sum_{d \in A} \frac{\over{c}^l}{ (cz + d)^l} = \left ( \sum_{c \in A_+} \sum_{d \in A} \frac{\over{c}}{cz + d} \right)^l,
\]
and so by analytic continuation this equality extends to $|t| < q^q$ provided that $z \in \Omega_1$.
\end{proof}

Recall that we have defined $\EE$ as $h d_2^{(1)}$. Corollary~5 from \cite{Pel_tau_recur_seq} shows that for $|t| < q^q$ we have the following series expansion
\[
	\EE = \sum_{c \in \AA_+} \over{c} u_c,
\]
where $u_c: \Omega \to \CCc$ is the function $u_c := e_{\carperiod A} (\carperiod c z)^{-1}$. 

As a special case of  Theorem \ref{sum_c_d_to_product} we obtain the following generalization, which was first conjectured in \cite[Remark~3.7]{Pet}.

\begin{theorem} \label{E^2} Let $1 \leq l \leq q$. For $(z, t) \in \Omega_1 \times B_{q^q}$, we have
\[
	\EE^l = \sum_{c \in \AA_+} \over{c}^l u_c^l.
\]
\end{theorem}

\begin{proof} This follows immediately from Corollary \ref{improved_convergence}.
\end{proof}

\begin{remark} The range $1 \leq l \leq q$ is natural because of properties of Goss polynomials, since in this range the $l$-th Goss polynomial is just $X^l$ (see \cite[(3.4)]{Gek_88_paper}). It is not difficult to find counterexamples to possible extensions of Theorem~\ref{E^2} if we go beyond $l = q$. 
\end{remark}

\begin{remark}
Theorem~\ref{E^2} also provides examples of deformations of Drinfeld modular forms with $A$-expansions. That is expansions of the form
\[
	\sum_{c \in A_+} a_c (t) G_n (u_c),
\]
where $a_c (t) \in A[t] = \Fq[\theta, t]$, $G_n$ is the $n$-th Goss polynomial of the lattice $\carperiod A$ as defined in \cite[Proposition~2.17]{Gos_eisenstein}. It is natural to wonder if there are more examples of these, just as in the case of Drinfeld modular forms (see \cite[Theorem~1.3]{Pet}). 

Computations with SAGE \cite{SAGE} suggest that this is indeed the case. For example, our computations suggest that for
\[	
	f_s = \sum_{c \in A_+} c^{1 + s(q-1)} u_c \in M_{2 + s (q-1), 1},
\]
we have
\[
	\mathbf{f}_s := f_s d_2 = \sum_{c \in A_+} \over{c} c^{s (q-1)} u_c,
\]
for $s = 1, \ldots, q, q+2, q+3, \ldots, q^2$. We hope to return to this topic in future work.
\end{remark}

Next, we turn to applications of Theorem \ref{E^2} to Drinfeld modular forms with $\AA$-expansions (see \cite{Lopez}, \cite{Pet}). We assume throughout that $1 \leq l \leq q$. For $\nu \in \NN$, define
\[
	f_{l, \nu} := \sum_{c \in \AA_+} c^{l q^\nu} u_c^l \in M_{l q^\nu + l, l}.
\]
We will use Theorem \ref{E^2} to give a recursive formula for $f_{l, \nu}$.

\begin{theorem} \label{recursive}
We have $f_{l, 1} = h^l$, $f_{l, 2} =  h^l g^{l q}$ and the recursive formula for $\nu \geq 2$
\[
	f_{l, \nu} = \left ( \frac{g^q}{h^{q-1}} f_{2, \nu-1}^q - \frac{[\nu -2]^{q^2}}{h^{q-1}} f_{2, \nu-2}^{q^2} \right)^l.
\]
\end{theorem}

\begin{proof} Let $\frob$ be the $q$th power Frobenius map acting on $\CCc((t))$. Then
\[
	\left ( (\frob \circ \tau^{-1})^{\nu} \EE^l \right)_{\mid_{t = \theta}} = f_{l, \nu}.
\]
Pellarin has shown (\cite[Proposition~9]{Pel_order_of_vanishing}) that
\[
	\EE = \frac{g^q}{h^{q-1}} \tau \EE - \frac{(u - \theta^{q^2})}{h^{q-1}} \tau^2 \EE.
\]
Therefore 
\[
	\EE^l = \left (  \frac{g^q}{h^{q-1}} \tau \EE - \frac{(t - \theta^{q^2})}{h^{q-1}} \tau^2 \EE \right)^l.
\]
Applying $(\frob \circ \tau^{-1})^\nu$ to both sides and plugging in $t = \theta$ finishes the proof.
\end{proof}

\begin{corollary} \label{application_to_A_exp} For $\nu \in \NN$, we have the eigenproduct identity of Drinfeld modular forms $f_{1, \nu}^l = f_{l, \nu}.$
\end{corollary}

\begin{proof} Indeed, both $f_{1, \nu}$ and $f_{l, \nu}$ are eigenforms according to \cite[Theorem~2.3]{Pet}.
\end{proof}


\bibliographystyle{amsplain}
\bibliography{references}


\end{document}